\documentclass[final]{amsart}

\newtheorem{thm}{Theorem}[section]
\newtheorem{lem}[thm]{Lemma}
\newtheorem{prop}[thm]{Proposition}

\theoremstyle{definition}
\newtheorem{defn}[thm]{Definition}

\theoremstyle{remark}
\newtheorem{rmk}[thm]{Remark}

\numberwithin{equation}{section}
\usepackage{mathrsfs}
\usepackage{hyperref}
\usepackage{url}
\usepackage{euscript}
\usepackage{amssymb}

\renewcommand{\labelenumi}{(\roman{enumi})}

\newcommand{\C}{\mathbb C}

\newcommand{\N}{\mathbb N}
	
\newcommand{\ms}[1]{\mathscr{#1}}
\newcommand{\mc}[1]{\mathcal{#1}}
\newcommand{\frk}[1]{\mathfrak{#1}}
\newcommand{\eu}[1]{\EuScript{#1}}

\begin{document}

\title[Quantum Rotatability]{Quantum rotatability}
\author{Stephen Curran}
\address{Department of Mathematics\\University of California at Berkeley\\Berkeley, CA 94720}
\email{\href{mailto:curransr@math.berkeley.edu}{curransr@math.berkeley.edu}}
\urladdr{\href{http://www.math.berkeley.edu/~curransr}{\url{http://www.math.berkeley.edu/~curransr}}}
\date{\today}
\begin{abstract}
In \cite{spekos}, K\"{o}stler and Speicher showed that de Finetti's theorem on exchangeable sequences has a free analogue if one replaces exchangeability by the stronger condition of invariance of the joint distribution under quantum permutations.  In this paper we study sequences of noncommutative random variables whose joint distribution is invariant under quantum orthogonal transformations.  We prove a free analogue of Freedman's characterization of conditionally independent Gaussian families, namely, the joint distribution of an infinite sequence of self-adjoint random variables is invariant under quantum orthogonal transformations if and only if the variables form an operator-valued free centered semicircular family with common variance.  Similarly, we show that the joint distribution of an infinite sequence of random variables is invariant under quantum unitary transformations if and only if the variables form an operator-valued free centered circular family with common variance.

We provide an example to show that, as in the classical case, these results fail for finite sequences.  We then give an approximation for how far the distribution of a finite quantum orthogonally invariant sequence is from the distribution of an operator-valued free centered semicircular family with common variance.
\end{abstract}

\maketitle

\section{Introduction}

The study of distributional symmetries has led to many deep structural results in probability.  The most well-known example is de Finetti's theorem on exchangeable sequences.  A sequence $(\xi_i)_{i \in \N}$ of random variables is called \textit{exchangeable} if the joint distribution of $(\xi_i)_{i \in \N}$ is invariant under finite permutations.  De Finetti's theorem states that an infinite exchangeable sequence of random variables is conditionally independent and identically distributed.  Another basic symmetry is \textit{rotatability}, defined as invariance of the joint distribution under orthogonal transformations.  In \cite{freedman}, Freedman showed that any infinite sequence of rotatable, real-valued random variables must form a conditionally independent centered Gaussian family with common variance.  Although these results fail for finite sequences, approximate results may still be obtained (see \cite{pd1}, \cite{pd2}, \cite{pd3}).  For a modern treatment of these and many related results, the reader is referred to the recent text of Kallenberg \cite{kallenberg}.

Exchangeability and rotatability are defined by distributional invariance under group actions.  In the noncommutative setting, group actions are typically replaced by coactions of quantum groups, and it is therefore natural to consider families of noncommutative variables whose joint distribution is invariant under coactions of quantum groups.  In particular, Wang introduced a noncommutative version of the permutation group $S_n$ in \cite{qpermutations}, called the \textit{quantum permutation group} $A_s(n)$, which leads to the condition of \textit{quantum exchangeability} for a sequence of noncommutative random variables.  K\"{o}stler and Speicher introduced this notion in \cite{spekos}, and showed that de Finetti's theorem has a natural free analogue: an infinite sequence of noncommutative random variables is quantum exchangeable if and only if the variables are freely independent and identically distributed with respect to a conditional expectation.  This was further studied in \cite{qexcalg}, where we extended this result to more general sequences and gave an approximation result for finite sequences.

In this paper, we consider sequences of noncommutative random variables whose joint distribution is invariant under quantum orthogonal transformations, in the sense of the \textit{quantum orthogonal group} $A_o(n)$ of Wang \cite{qorthunit}.  Our main result is the following free analogue of Freedman's characterization of conditionally independent Gaussian families:

\begin{thm}\label{infrot}
Let $(x_i)_{i \in \N}$ be a sequence of self-adjoint random variables in the W$^*$-probability space $(M,\varphi)$.  Then the following are equivalent:
\begin{enumerate}
\item The joint distribution of $(x_i)_{i \in \N}$ is invariant under quantum orthogonal transformations.
\item There is a W$^*$-subalgebra $1 \in B \subset M$, and a conditional expectation $E:W^*(\{x_i:i \in \N\}) \to B$ which preserves $\varphi$ such that $\{x_i:i \in \N\}$ form a $B$-valued freely independent centered semicircular family with common variance, with respect to $E$.
\end{enumerate}
\end{thm}
\noindent It is well-known that the semicircular distribution plays the role of the Gaussian distribution in free probability, in particular it is the limit distribution of the free central limit theorem \cite{vdn}.  Note that the free independence is not part of our assumptions, but is instead a result of the invariance condition.  If one assumes \textit{a priori} that the variables are freely independent, then it is known that the variables are centered semicircular with common variance if and only if their joint distribution is invariant under usual orthogonal transformations (\cite{nica}).

As in the classical case, Theorem \ref{infrot} fails for finite sequences (we provide an example in \ref{finex}).  However, we can give the following approximation:

\begin{thm}\label{finrot}
Let $(x_1,\dotsc,x_n)$ be a sequence of self-adjoint random variables in the W$^*$-probability space $(M,\varphi)$ whose joint distribution is invariant under quantum orthogonal transformations.  Then there is a W$^*$-subalgebra $1 \in B \subset M$, and a $\varphi$-preserving conditional expectation $E:W^*(\{x_i:i \in \N\}) \to B$ such that if $s_1,\dotsc,s_n$ is a $B$-valued free centered semicircular family with common variance $\eta:B \to B$ defined by
\begin{equation*}
 \eta(b) = E[x_1bx_1],
\end{equation*}
then for any $k \in \N$, $1 \leq i_1,\dotsc,i_{2k+1} \leq n$ and $b_0,\dotsc,b_{2k+1} \in B$ such that $\|b_l\| \leq 1$ for $1 \leq l \leq 2k$, we have
\begin{equation*}
\bigl\|E[b_0x_{i_1}\dotsb x_{i_{2k}}b_{2k}] - E[b_0s_{i_1}\dotsb s_{i_{2k}}b_{2k}]\bigr\| \leq \frac{D_k}{n}\|x_1\|^{2k}
\end{equation*}
where $D_k$ is a universal constant which depends only on $k$, and
\begin{equation*} 
 E[b_0x_{i_1}\dotsb x_{i_{2k+1}}b_{2k+1}] = E[b_0s_{i_1}\dotsb s_{i_{2k+1}}b_{2k+1}] = 0.
\end{equation*}

\end{thm}

Wang also introduced a noncommutative version $A_u(n)$ of the unitary group $U_n$ in \cite{qorthunit}.  For quantum unitarily invariant sequences of noncommutative random variables, similar results hold if one replaces the semicircular distribution by the circular distribution, which is the analogue in free probability of the complex Gaussian distribution.  In particular, we will prove the following characterization of operator-valued free circular families:

\begin{thm}\label{infunit}
Let $(x_i)_{i \in \N}$ be an infinite sequence of noncommutative random variables in the W$^*$-probability space $(M,\varphi)$.  Then the following are equivalent:
\begin{enumerate}
\item The joint distribution of $(x_i)_{i \in \N}$ is invariant under quantum unitary transformations.
\item There is a W$^*$-subalgebra $1 \in B \subset M$ and a conditional expectation $E:W^*(\{x_i:i \in \N\}) \to B$ such that $(x_i)_{i \in \N}$ form a $B$-valued free centered circular family with common variance, with respect to $E$.
\end{enumerate}
\end{thm}

Our approach is similar to that presented in \cite{qexcalg} for quantum exchangeable sequences, and is based on the compact quantum group structure of $A_o(n)$ and $A_u(n)$.  We find that for a sequence of random variables in a W$^*$-probability space whose joint distribution is quantum orthogonally (resp. unitarily) invariant, there is a natural conditional expectation given by integrating a coaction of $A_o(n)$ (resp. $A_u(n)$) with respect to the Haar state.  Using the formula for the Haar states on $A_o(n)$ and $A_u(n)$, computed by Banica and Collins in \cite{bc}, we give an explicit form for this conditional expectation.  The structure which appears in these computations is the operator-valued moment-cumulant formula of Speicher \cite{memoir}.

The paper is organized as follows:  In Section 2 we recall the basic definitions and results from free probability, and introduce the quantum orthogonal and unitary groups.  In Section 3, we define quantum rotatability for finite sequences and prove Theorem \ref{finrot}.  Section 4 contains the proof of Theorem \ref{infrot}, and an example which shows that this result fails for finite sequences.  In Section 5, we consider quantum unitarily invariant sequences and prove Theorem \ref{infunit}.
\section{Preliminaries and Notations}

\noindent\textbf{Notations}.  Given an index set $I$, we denote by $\ms Q_I$ the $*$-algebra of noncommutative polynomials $\ms Q_I = \C\langle t_i,t_i^*: i \in I\rangle$.  The universal property of $\ms Q_I$ is that given any unital $*$-algebra $A$ and a family $(x_i)_{i \in I}$ of elements in $A$, there is a unique unital $*$-homomorphism $\mathrm{ev}_x:\ms Q_I \to A$ such that $\mathrm{ev}_x(t_i) = x_i$ for each $i \in I$.  We will also denote this map by $q \mapsto q(x)$ for $q \in \ms Q_I$.  

We define $\ms P_I$ to be the quotient of $\ms Q_I$ by the relations $t_i = t_i^*$ for $i \in I$.  The universal property of $\ms P_I$ is that whenever $A$ is a unital $*$-algebra and $(x_i)_{i \in I}$ is a family of self-adjoint elements in $A$, there is a unique homomorphism from $\ms P_I$ into $A$, which we also denote $\mathrm{ev}_x$, such that $\mathrm{ev}_x(t_i) = x_i$.  We will also denote this map by $p \mapsto p(x)$ for $p \in \ms P_I$.

We will mostly be interested in the case that $I = \{1,\dotsc,n\}$, in which case we denote $\ms Q_I$ and $\ms P_I$ by $\ms Q_n$ and $\ms P_n$, and $I = \N$ in which case we denote $\ms Q_I = \ms Q_\infty$, $\ms P_I = \ms P_\infty$.

\noindent\textbf{Free Probability.} We begin by recalling some basic notions from free probability, the reader is referred to \cite{vdn}, \cite{ns} for further information.

\begin{defn}\hfill
\begin{enumerate}
 \item A \textit{noncommutative probability space} is a pair $(A,\varphi)$, where $A$ is a unital $*$-algebra and $\varphi$ is a state.  
\item A noncommutative probability space $(M,\varphi)$, where $M$ is a von Neumann algebra and $\varphi$ is a faithful normal state, is called a \textit{W$^*$-probability space}.  We do not require $\varphi$ to be tracial.
\end{enumerate}

\end{defn}

\begin{defn}
The \textit{joint $*$-distribution} of a family $(x_i)_{i \in I}$ of random variables  in a noncommutative probability space $(A,\varphi)$ is the linear functional $\varphi_x$ on $\mathscr Q_I$ defined by $\varphi_x(q) = \varphi(q(x))$.  $\varphi_x$ is determined by the \textit{moments}
\begin{equation*}
 \varphi_x(t_{i_1}^{d_1}t_{i_2}^{d_2}\dotsb t_{i_k}^{d_k}) = \varphi(x_{i_1}^{d_1}x_{i_2}^{d_2}\dotsb x_{i_k}^{d_k}),
\end{equation*}
where $i_1,\dotsc,i_k \in I$ and $d_1,\dotsc,d_k \in \{1,*\}$.  When $x_i = x_i^*$ for each $i \in I$, $\varphi_x$ factors through $\ms P_I$ and we then use $\varphi_x$ to denote the induced linear functional on $\ms P_I$.
\end{defn}

\begin{rmk}
These definitions have natural ``operator-valued'' extensions given by replacing $\C$ by a more general algebra of scalars.  This is the right setting for the notion of freeness with amalgamation, which is the analogue of conditional independence in free probability.
\end{rmk}

\begin{defn}
A \textit{$B$-valued probability space} $(A,E)$ consists of a unital $*$-algebra $A$, a $*$-subalgebra $1 \in B \subset A$, and a conditional expectation $E:A \to B$, i.e. $E$ is a linear map such that $E[1] = 1$ and
\begin{equation*}
 E[b_1ab_2] = b_1E[a]b_2
\end{equation*}
for all $b_1,b_2 \in B$ and $a \in A$. 
\end{defn}

\begin{defn}
Let $(A,E)$ be a $B$-valued probability space and $(x_i)_{i \in I}$ a family of random variables in $A$.
\begin{enumerate}
\item  We let $B \langle t_i, t_i^*:i \in I\rangle$ denote the $*$-algebra of noncommutative polynomials with coefficients in $B$.  There is a unique $*$-homomorphism from $B \langle t_i,t_i^*: i \in I \rangle$ into $A$ which is the identity on $B$ and sends $t_i$ to $x_i$, which we denote by $p \mapsto p(x)$.
\item Likewise we let $B \langle t_i :i \in I \rangle$ denote the $*$-algebra of noncommutative polynomials with coefficients in $B$ and self-adjoint generators indexed by $I$.  If $x_i = x_i^*$ for each $i \in I$, then the homomorphism from (i) factors through $B \langle t_i: i \in I \rangle$, we will still denote this by $p \mapsto p(x)$.
\item The \textit{$B$-valued joint distribution} of the family $(x_i)_{i \in I}$ is the linear map $E_x:B\langle t_i,t_i^*:i \in I \rangle \to B$ defined by $E_x(p) = E[p(x)]$.  $E_x$ is determined by the \textit{$B$-valued moments}
\begin{equation*}
 E_x[b_0t_{i_1}^{d_1}\dotsb t_{i_k}^{d_k}b_{k}] = E[b_0x_{i_1}^{d_1}\dotsb x_{i_k}^{d_k}b_k]
\end{equation*}
for $b_0,\dotsc,b_k \in B$, $i_1,\dotsc,i_k \in I$ and $d_1,\dotsc,d_k \in \{1,*\}$.  If $x_i = x_i^*$ for every $i \in I$, then $E_x$ factors through $B \langle t_i: i \in I\rangle$ and we will then use $E_x$ to denote the induced linear map from $B \langle t_i:i \in I\rangle$ to $B$.

\item The family $(x_i)_{i \in I}$ is called \textit{free with respect to $E$} or \textit{free with amalgamation over $B$} if
\begin{equation*}
 E[p_1(x_{i_1},x_{i_1}^*)\dotsb p_k(x_{i_k},x_{i_k}^*)] = 0
\end{equation*}
whenever $p_1,\dotsc,p_k \in B \langle t,t^* \rangle$, $i_1,\dotsc,i_k \in I$, $i_1 \neq \dotsb \neq i_k$ and 
\begin{equation*} \\
E[p_l(x_{i_l},x_{i_l}^*)] = 0
\end{equation*}
for $1 \leq l \leq k$.

\end{enumerate}
\end{defn}

\begin{rmk}
Free independence with amalgamation has a rich combinatorial theory, developed by Speicher in \cite{memoir}.  The basic objects are non-crossing set partitions and free cumulants, which we will now recall.  For further information on the combinatorial aspects of free probability, the reader is referred to \cite{ns}.
\end{rmk}

\begin{defn}\hfill

\begin{enumerate}
\item A \textit{partition} $\pi$ of a set $S$ is a collection of disjoint, non-empty sets $V_1,\dotsc,V_r$ such that $V_1 \cup \dotsb \cup V_r = S$.  $V_1,\dotsc,V_r$ are called the \textit{blocks} of $\pi$, and we set $|\pi| = r$. If $s,t \in S$ are in the same block of $\pi$, we write $s \sim_\pi t$. The collection of partitions of $S$ will be denoted $\mc P(S)$, or in the case that $S =\{1,\dotsc,k\}$ by $\mc P(k)$.
\item Given $\pi,\sigma \in \mc P(S)$, we say that $\pi \leq \sigma$ if each block of $\pi$ is contained in a block of $\sigma$.  There is a least element of $\mc P(S)$ which is larger than both $\pi$ and $\sigma$, which we denote by $\pi \vee \sigma$.  
\item If $S$ is ordered, we say that $\pi \in \mc P(S)$ is \textit{non-crossing} if whenever $V,W$ are blocks of $\pi$ and $s_1 < t_1 < s_2 < t_2$ are such that $s_1,s_2 \in V$ and $t_1,t_2 \in W$, then $V = W$.  The set of non-crossing partitions of $S$ is denoted by $NC(S)$, or by $NC(k)$ in the case that $S = \{1,\dotsc,k\}$.

\item The non-crossing partitions can also be defined recursively, a partition $\pi \in \mc P(S)$ is non-crossing if and only if it has a block $V$ which is an interval, such that $\pi \setminus V$ is a non-crossing partition of $S \setminus V$.
\item  Given $i_1,\dotsc,i_k$ in some index set $I$, we denote by $\ker \mathbf i$ the element of $\mc P(k)$ whose blocks are the equivalence classes of the relation
\begin{equation*}
 s \sim t \Leftrightarrow i_s= i_t.
\end{equation*}
Note that if $\pi \in \mc P(k)$, then $\pi \leq \ker \mathbf i$ is equivalent to the condition that whenever $s$ and $t$ are in the same block of $\pi$, $i_s$ must equal $i_t$.
\item If $\pi \in NC(k)$ is such that every block of $\pi$ has exactly 2 elements, we call $\pi$ a \textit{non-crossing pair partition}.  We let $NC_2(k)$ denote the set of non-crossing pair partitions of $\{1,\dotsc,k\}$.
\item Let $d_1,\dotsc,d_k \in \{1,*\}$.  We let
\begin{equation*}
NC_2^{\mathbf d}(k) = \{\pi \in NC_2(k): s \sim_\pi t \Rightarrow d_s \neq d_t\}.
\end{equation*}
\end{enumerate}
\end{defn}

\begin{defn}
Let $(A,E)$ be a $B$-valued probability space. 
\begin{enumerate}
\renewcommand{\labelenumi}{(\roman{enumi})}
\item For each $k \in \N$, let $\rho^{(k)}:A^{\otimes_B k} \to B$ be a linear map (the tensor product is with respect to the natural $B-B$ bimodule structure on $A$).  For $n \in \N$ and $\pi \in NC(n)$, we define a linear map $\rho^{(\pi)}: A^{\otimes_B n} \to B$ recursively as follows.  If $\pi$ has only one block, we set
\begin{equation*}
 \rho^{(\pi)}[a_1 \otimes \dotsb \otimes a_n] = \rho^{(n)}(a_1 \otimes \dotsb \otimes a_n)
\end{equation*}
for any $a_1,\dotsc,a_n \in A$.  Otherwise, let $V = \{l+1,\dotsc,l+s\}$ be an interval of $\pi$.  We then define, for any $a_1,\dotsc,a_n \in A$, 
\begin{equation*}
 \rho^{(\pi)}[a_1 \otimes \dotsb \otimes a_n] = \rho^{(\pi \setminus V)}[a_1 \otimes \dotsb \otimes a_{l}\rho^{(s)}(a_{l+1} \otimes \dotsb \otimes a_{l+s}) \otimes \dotsb \otimes a_n].
\end{equation*}

\item For $k \in \N$, define the \textit{$B$-valued moment functions} $E^{(k)}:A^{\otimes_B k} \to B$ by
\begin{equation*}
 E^{(k)}[a_1 \otimes \dotsb \otimes a_k] = E[a_1\dotsb a_k].
\end{equation*}

\item The \textit{$B$-valued cumulant functions} $\kappa_E^{(k)}:A^{\otimes_B k} \to B$ are defined recursively for $\pi \in NC(k)$, $k \geq 1$, by the \textit{moment-cumulant formula}: for each $n \in \N$ and $a_1,\dotsc,a_n \in A$ we have
\begin{equation*}
 E[a_1\dotsb a_n] = \sum_{\pi \in NC(n)} \kappa_E^{(\pi)}[a_1 \otimes \dotsb \otimes a_n].
\end{equation*}
Note that the right hand side of this formula is equal to $\kappa_E^{(n)}(a_1 \otimes \dotsb \otimes a_n)$ plus lower order terms, and hence can be recursively solved for $\kappa_E^{(n)}$.  The moment and cumulant functions are related by the following formula (\cite{memoir}):
\begin{equation*}
 \kappa_E^{(\pi)}[a_1 \otimes \dotsb \otimes a_n] = \sum_{\substack{\sigma \in NC(n)\\ \sigma \leq \pi}} \mu_n(\sigma,\pi)E^{(\sigma)}[a_1 \otimes \dotsb \otimes a_n],
\end{equation*}
where $\mu_n$ is the \textit{M\"{o}bius function} on the partially ordered set $NC(n)$.  
\end{enumerate}
\end{defn}

\begin{rmk}
The key relation between $B$-valued cumulant functions and free independence with amalgamation is the following result of Speicher, which characterizes freeness in terms of the ``vanishing of mixed cumulants''.
\end{rmk}

\begin{thm}\textnormal{(\cite{memoir})}\label{vancum}  Let $(A,E)$ be a $B$-valued probability space and $(x_i)_{i \in I}$ a family of random variables in $A$.  Then the family $(x_i)_{i \in I}$ is free with amalgamation over $B$ if and only if
\begin{equation*}
 \kappa_{E}^{(\pi)}[x_{i_1}^{d_1}b_1 \otimes \dotsb \otimes x_{i_k}^{d_k}b_k] = 0
\end{equation*}
whenever $i_1,\dotsc,i_k \in I$, $b_1,\dotsc, b_k \in B$, $d_1,\dotsc,d_k \in \{1,*\}$ and $\pi \in NC(k)$ is such that $\pi \not\leq \ker \mathbf i$.
\end{thm}

\noindent \textbf{Operator-valued semicircular and circular families.}  We now recall the combinatorial descriptions of semicircular and circular random variables, which are the free analogues of real and complex Gaussian random variables, respectively.  Operator-valued semicircular random variables were first considered by Voiculescu in \cite{opadd}, where they were shown to be the limiting distribution of an operator-valued free central limit theorem.  The combinatorial description which we now present is due to Speicher \cite{memoir}, where they are referred to as \textit{$B$-Gaussians}.  Operator-valued circular random variables can be viewed as a special case of Speicher's $B$-Gaussians, the definition given here is from \cite{bcircular}.

\begin{defn} 

Let $(A,E)$ be a $B$-valued probability space. 
\begin{enumerate}
\renewcommand{\labelenumi}{(\roman{enumi})}
\item A family $(s_i)_{i \in I}$ of self-adjoint random variables in $A$ is said to form a \textit{$B$-valued free centered semicircular family} if for any $i_1,\dotsc,i_k \in I$ and $b_0,\dotsc,b_k \in B$, we have
\begin{equation*}
 \kappa_E^{(\pi)}[b_0s_{i_1}b_1 \otimes \dotsb \otimes s_{i_k}b_k] = 0
\end{equation*}
unless $\pi \in NC_2(k)$ and $\pi \leq \ker \mathbf i$.  In particular, the family $(s_i)_{i \in I}$ is free with amalgamation over $B$ by Theorem \ref{vancum}.  The $B$-valued joint distribution of the family $(s_i)_{i \in I}$ is then determined by the linear maps $\eta_i:B \to B$, called the \textit{variances}, defined by
\begin{equation*}
 \eta_i(b) = \kappa_E^{(2)}[s_ib \otimes s_i] = E[s_i b s_i].
\end{equation*}
\item A family $(c_i)_{i \in I}$ of (non self-adjoint) random variables in $A$ is said to form a \textit{$B$-valued free centered circular family} if for any $i_1,\dotsc,i_k \in I$, $b_0,\dotsc,b_k \in B$ and $d_1,\dotsc,d_k \in \{1,*\}$ we have
\begin{equation*}
 \kappa_E^{(\pi)}[b_0c_{i_1}^{d_1}b_1 \otimes \dotsb \otimes c_{i_k}^{d_k}b_k] = 0
\end{equation*}
unless $\pi \in NC_2^{\mathbf d}(k)$ and $\pi \leq \ker \mathbf i$.  The $B$-valued joint distribution of $(c_i)_{i \in I}$ is then defined by the linear maps $\eta_i,\theta_i:B \to B$, also called the \textit{variances}, defined by
\begin{align*}
 \eta_i(b) &= \kappa_E^{(2)}[c_i^* b \otimes c_i]\\
\theta_i(b) &= \kappa_E^{(2)}[c_ib \otimes c_i^*].
\end{align*}
\end{enumerate}
\end{defn}

\begin{rmk}
We will use the following result, which is immediate from the definitions above and the moment-cumulant formula, in our proofs of Theorems \ref{infrot} and \ref{infunit}.
\end{rmk}

\begin{prop} \label{semcircexp}Let $(A,E)$ be a $B$-valued probability space.
\begin{enumerate}
 \item Let $(x_i)_{i \in \N}$ be a sequence of self-adjoint elements in $A$.  Then $(x_i)_{i \in \N}$ form a $B$-valued free centered semicircular family with common variance if and only if
\begin{align*}
 E[b_0x_{i_1}\dotsb x_{i_{2k}}b_{2k}] &= \sum_{\substack{\pi \in NC_2(2k)\\ \pi \leq \ker \mathbf i}} \kappa_E^{(\pi)}[b_0x_{1}b_1 \otimes \dotsb \otimes x_{1}b_{2k}]\\
E[b_0x_{i_1}\dotsb x_{i_{2k+1}}b_{2k+1}] &= 0
\end{align*}
for any $i_1,\dotsc,i_{2k+1} \in \N$ and $b_0,\dotsc,b_{2k+1} \in B$.
\item Let $(x_i)_{i \in \N}$ be a sequence in $A$.  Then $(x_i)_{i \in \N}$ form a $B$-valued free centered circular family with common variance if and only if
\begin{align*}
 E[b_0x_{i_1}^{d_1}\dotsb x_{i_{2k}}^{d_{2k}}b_{2k}] &= \sum_{ \substack{\pi \in NC_2^{\mathbf d}\\ \pi \leq \ker \mathbf i}} \kappa_E^{(\pi)}[b_0x_{1}^{d_1}b_1 \otimes \dotsb \otimes x_{1}^{d_{2k}}b_{2k}]\\
E[b_0x_{i_1}^{d_1}\dotsb x_{i_{2k+1}}^{d_{2k+1}}b_{2k+1}] &= 0
\end{align*}
for any $i_1,\dotsc,i_{2k+1} \in \N$, $b_0,\dotsc,b_{2k+1} \in B$ and $d_1,\dotsc,d_{2k+1} \in \{1,*\}$.
\end{enumerate}

\end{prop}

\noindent\textbf{Quantum Orthogonal and Unitary Groups.}  We recall the definitions of the universal quantum groups $A_o(n)$ and $A_u(n)$ from \cite{qorthunit}.  For further information about these compact quantum groups, see \cite{bbc}, \cite{bc}.   

\begin{defn} \hfill
\begin{enumerate}
 \item The \textit{quantum orthogonal group} $A_o(n)$ is the universal unital C$^*$-algebra with generators $\{u_{ij}: 1 \leq i,j \leq n\}$ and relations such that $u = (u_{ij}) \in M_n(A_o(n))$ is orthogonal, i.e. $u_{ij} = u_{ij}^*$ and $u^t = u^{-1}$.  In particular, we have
\begin{align*}
 \sum_{k = 1}^n u_{ki}u_{kj} = \delta_{ij}1_{A_o(n)} = \sum_{k =1}^n u_{ik}u_{jk}.
\end{align*}
$A_o(n)$ is a compact quantum group, with comultiplication, counit and antipode given by the formulas
\begin{align*}
 \Delta(u_{ij}) &= \sum_{k = 1}^n u_{ik} \otimes u_{kj}\\
\epsilon(u_{ij}) &= \delta_{ij}\\
S(u_{ij}) &= u_{ji}.
\end{align*}
The existence of the maps above are given by the universal property of $A_o(n)$.  $A_o(n)$ has a canonical dense Hopf $*$-algebra $\mc A_o(n)$, which is the $*$-algebra generated by $\{u_{ij}: 1 \leq i,j \leq n\}$.  A fundamental theorem of Woronowicz \cite{woronowicz} gives the existence of unique state $\psi_n:A_o(n) \to \C$, called that \textit{Haar state}, which is left and right invariant in the sense that
\begin{equation*}
 (\mathrm{id} \otimes \psi_n)\Delta_n(a) = \psi_n(a)1_{A_o(n)} = (\psi_n \otimes \mathrm{id})\Delta_n(a)
\end{equation*}
for any $a \in A_o(n)$.  We will denote the GNS representation for the Haar state by $\pi_{\psi_n}$, and we set $\frk A_o(n) = \pi_{\psi_n}(A_o(n))''$, which has a natural Hopf von Neumann algebra structure.  
\item The \textit{quantum unitary group} $A_u(n)$ is the universal C$^*$-algebra with generators $\{v_{ij}:1 \leq i,j \leq n\}$ and relations such that the matrix $(v_{ij}) \in M_n(A_o(n))$ is unitary.  More explicitly, the relations are
\begin{equation*}
 \sum_{k = 1}^n v_{ki}^*v_{kj} = \delta_{ij}1_{A_u(n)} = \sum_{k =1}^n v_{ik}v_{jk}^*.
\end{equation*}
$A_u(n)$ is a compact quantum group with comultiplication, counit and antipode given by the formulas
\begin{align*}
 \Delta(v_{ij}) &= \sum_{k = 1}^n v_{ik} \otimes v_{kj}\\
\epsilon(v_{ij}) &= \delta_{ij}\\
S(v_{ij}) &= v_{ji}^*.
\end{align*}
As for $A_o(n)$, the existence of these maps is given by the universal property of $A_u(n)$. We let $\mc A_u(n)$ denote the canonical dense Hopf $*$-algebra generated by $\{v_{ij}:1 \leq i,j \leq n\}$.  We will also use $\psi_n$ to denote the Haar state on $A_u(n)$, and $\pi_{\psi_n}$ the corresponding GNS representation.  We define $\frk A_u(n)$ to be the Hopf von Neumann algebra $\frk A_u(n) = \pi_{\psi_n}(A_u(n))''$.  

\end{enumerate}

\end{defn}

\begin{rmk}
If one adds commutativity to the above relations, then the resulting universal C$^*$-algebras are simply the continuous functions on the orthogonal and unitary groups, respectively.  We will need the following formulas for the Haar states on $A_o(n)$ and $A_u(n)$, which were computed by Banica and Collins in \cite{bc}.
\end{rmk}

\begin{rmk}\label{haar} \textbf{The Haar States.}  
\begin{enumerate}
 \item For $k \in \N$, let $G_{kn}$ be the matrix with entries indexed by non-crossing pair partitions in $NC_2(2k)$ defined by
\begin{equation*}
 G_{kn}(\pi,\sigma) = n^{|\pi \vee \sigma|},
\end{equation*}
where the join is taken in the lattice $\mc P(2k)$.  For $n \geq 2$, $G_{kn}$ is invertible and the \textit{Weingarten matrix} $W_{kn}$ is then defined as its inverse.  The Haar state on $A_o(n)$ is determined by the formula
\begin{align*}
 \psi_n(u_{i_1j_1}\dotsb u_{i_{2k}j_{2k}}) &= \sum_{\substack{\pi,\sigma \in NC_2(2k)\\ \pi \leq \ker \mathbf i\\ \sigma \leq \ker \mathbf j}} W_{kn}(\pi,\sigma)\\
\psi_n(u_{i_1j_1}\dotsb u_{i_{2k+1}j_{2k+1}}) &= 0.
\end{align*}
In particular, note that
\begin{equation*}
 \psi_n(u_{i_1j}u_{i_2j}) = \frac{1}{n}\delta_{i_1i_2}
\end{equation*}
for any $1 \leq i_1,i_2,j \leq n$.  The key fact about $W_{kn}$ which we will need is the following asymptotic estimate:
\begin{equation*}
 n^kW_{kn}(\pi,\sigma) = \delta_{\pi\sigma} + O(n^{-1}).
\end{equation*}
This follows from the power series expansion for $W_{kn}$ computed in \cite[Proposition 7.2]{bc}.
\item Let $d_1,\dotsc,d_{2k} \in \{1,*\}$.  We then let $G_{\mathbf dn}$ to be the matrix with entries indexed by $NC_2^{\mathbf d}(2k)$, defined by
\begin{equation*}
 G_{\mathbf dn}(\pi,\sigma) = n^{|\pi \vee \sigma|},
\end{equation*}
where the join is taken in the lattice $\mc P(2k)$.  We likewise define a Weingarten matrix $W_{\mathbf d n}$ to be the inverse of $G_{\mathbf d n}$, which exists for $n \geq 2$.  The Haar state on $A_u(n)$ is then determined by the formula
\begin{align*}
 \psi_n(v_{i_1j_1}^{d_1}\dotsb v_{i_{2k}j_{2k}}^{d_{2k}}) &= \sum_{\substack{\pi,\sigma \in NC_2^{\mathbf d}(2k)\\ \pi \leq \ker \mathbf i\\ \sigma \leq \ker \mathbf j}} W_{\mathbf d n}(\pi,\sigma)\\
\psi_n(v_{i_1j_1}^{d_1}\dotsb v_{i_{2k+1}j_{2k+1}}^{d_{2k+1}}) &= 0.
\end{align*}
We will need the following asymptotic estimate on $W_{\mathbf d n}$:
\begin{equation*}
 n^kW_{\mathbf d n}(\pi,\sigma) = \delta_{\pi\sigma} + O(n^{-1}).
\end{equation*}
This may be proved similarly to \cite[Proposition 7.2]{bc}, or by using the approach found in \cite[Lemma 4.12]{qexcalg}.
\end{enumerate}

\end{rmk}

\section{Finite quantum rotatable sequences}

\begin{rmk}Let $\alpha_n:\ms P_n \to \ms P_n \otimes \mc A_o(n)$ be the unique unital homomorphism determined by
\begin{equation*}
 \alpha_n(t_j) = \sum_{i=1}^n t_i \otimes u_{ij}.
\end{equation*}
It is easy to see that $\alpha_n$ is a right coaction of the Hopf $*$-algebra $\mc A_o(n)$ on $\ms P_n$, i.e.
\begin{align*}
 (\mathrm{id} \otimes \Delta) \circ \alpha_n &= (\alpha_n \otimes \mathrm{id}) \circ \alpha_n\\
\intertext{and}
(\mathrm{id} \otimes \epsilon) \circ \alpha_n &= \mathrm{id}.
\end{align*}

\end{rmk}

\begin{defn}
Let $(x_1,\dotsc,x_n)$ be a sequence of self-adjoint random variables in the noncommutative probability space $(A,\varphi)$.  We say that the distribution $\varphi_x$ is \textit{invariant under quantum orthogonal transformations}, or that the sequence $(x_1,\dotsc,x_n)$ is \textit{quantum orthogonally invariant} or \textit{quantum rotatable}, if $\varphi_x$ is invariant under the coaction $\alpha_n$, i.e.
\begin{equation*}
 (\varphi_x \otimes \mathrm{id})\alpha_n(p) = \varphi_x(p)1_{A_o(n)}
\end{equation*}
for all $p \in \ms P_n$.  
\end{defn}

\begin{rmk}\textit{Remarks.} \label{qrotrmk}
\begin{enumerate}
 \item More explicitly, the sequence $(x_1,\dotsc,x_n)$ is quantum rotatable if for any $1 \leq j_1,\dotsc,j_k \leq n$ we have
\begin{equation*}
 \sum_{1 \leq i_1,\dotsc,i_k \leq n} \varphi(x_{i_1}\dotsb x_{i_k})u_{i_1j_1}\dotsb u_{i_kj_k} = \varphi(x_{j_1}\dotsb x_{j_k})  1
\end{equation*}
as an equality in $A_o(n)$.
\item By the universal property of $A_o(n)$, the sequence $(x_1,\dotsc,x_n)$ is quantum rotatable if and only if the equation in (i) holds for any family $\{u_{ij}: 1 \leq i,j \leq n\}$ of self-adjoint elements in a unital C$^*$-algebra $B$ such that $(u_{ij}) \in M_n(B)$ is an orthogonal matrix.
\item For $1 \leq i,j \leq n$, define $f_{ij} \in C(O_n)$ by $f_{ij}(T) = T_{ij}$ for $T \in O(n)$.  The matrix $(f_{ij}) \in M_n(C(O_n))$ is orthogonal and the equation in (i) becomes 
\begin{equation*}
 \varphi(x_{j_1}\dotsb x_{j_k})1_{C(O_n)} = \sum_{1 \leq i_1,\dotsc,i_k \leq n} \varphi(x_{i_1}\dotsb x_{i_k})f_{i_1j_1}\dotsb f_{i_kj_k}.
\end{equation*}
It follows that for any $T \in O_n$,
\begin{align*}
 \varphi(x_{j_1}\dotsb x_{j_k}) &= \sum_{1 \leq i_1,\dotsc,i_k \leq n} \varphi(x_{i_1}\dotsb x_{i_k})T_{i_1j_1}\dotsb T_{i_kj_k}\\
&= \varphi(T(x)_{j_1}\dotsb T(x)_{j_k}),
\end{align*}
where $T(x)$ is the sequence obtained by applying $T$ to $(x_1,\dotsc,x_n)$ in the obvious way.  So quantum orthogonal invariance implies orthogonal invariance.

\item By taking $\{u_{ij}: 1\leq i,j \leq n\}$ to be the generators of the quantum permutation group $A_s(n)$, it follows from (ii) that quantum rotatability implies quantum exchangeability as defined in \cite{spekos}.
\end{enumerate}

\end{rmk}

\begin{rmk}
First we will show that operator-valued free centered semicircular families with common variance are quantum rotatable.  This holds in a purely algebraic context.  The proof is along the same lines as \cite[Proposition 3.1]{spekos}.
\end{rmk}

\begin{prop}\label{semrot}
Let $(A,\varphi)$ be a noncommutative probability space, $1 \in B \subset A$ a subalgebra and $E:B \to A$ a conditional expectation which preserves $\varphi$.  Suppose that $s_1,\dotsc,s_n$ form a $B$-valued free centered semicircular family with common variance.  Then the sequence $(s_1,\dotsc,s_n)$ is quantum rotatable.
\end{prop}

\begin{proof}
Let $1 \leq j_1,\dotsc,j_{2k} \leq n$, then
\begin{multline*}
 \sum_{1 \leq i_1,\dotsc,i_{2k} \leq n} \varphi(s_{i_1}\dotsb s_{i_{2k}})u_{i_1j_1}\dotsb u_{i_{2k}j_{2k}}\\
= \sum_{1 \leq i_1,\dotsc,i_{2k} \leq n} \varphi(E[s_{i_1}\dotsb s_{i_{2k}}])u_{i_1j_1}\dotsb u_{i_{2k}j_{2k}}\\
= \sum_{1 \leq i_1,\dotsc,i_{2k} \leq n} \sum_{\substack{\pi \in NC_2(2k)\\ \pi \leq \ker \mathbf i}} \varphi(\kappa_E^{(\pi)}[s_{i_1} \otimes \dotsb \otimes s_{i_{2k}}])u_{i_1j_1}\dotsb u_{i_{2k}j_{2k}}
\end{multline*}
Since the variables have common variance, given $\pi \in NC_2(2k)$ the value of $\kappa_E^{(\pi)}[s_{i_1} \otimes \dotsb \otimes s_{i_{2k}}]$ is the same for any $1 \leq i_1,\dotsc,i_{2k} \leq n$ such that $\pi \leq \ker \mathbf i$, we denote this common value by $\kappa_E^{(\pi)}$.  We then have
\begin{multline*}
 \sum_{1 \leq i_1,\dotsc,i_{2k} \leq n} \varphi(s_{i_1}\dotsb s_{i_{2k}})u_{i_1j_1}\dotsb u_{i_{2k}j_{2k}} \\
= \sum_{\pi \in NC_2(2k)} \varphi(\kappa_E^{(\pi)})\sum_{\substack{1 \leq i_1,\dotsc,i_{2k} \leq n\\ \pi \leq \ker \mathbf i}} u_{i_1j_1}\dotsb u_{i_{2k}j_{2k}}.
\end{multline*}
We now claim that for any $\pi \in NC_2(2k)$ and $1 \leq j_1,\dotsc,j_{2k} \leq n$, we have
\begin{equation*}
 \sum_{\substack{1 \leq i_1,\dotsc,i_{2k} \leq n\\ \pi \leq \ker \mathbf i}} u_{i_1j_1}\dotsb u_{i_{2k}j_{2k}} = \begin{cases} 1_{A_o(n)}, & \pi \leq \ker \mathbf j\\ 0, & \text{otherwise}\end{cases}.
\end{equation*}
We prove this by induction, the case $k=1$ is simply the orthogonality relation.  Suppose $k > 1$, let $\pi \in NC_2(2k)$ and let $V = \{l,l+1\}$ be an interval of $\pi$.  Then
\begin{multline*}
\sum_{\substack{1 \leq i_1,\dotsc,i_{2k} \leq n\\ \pi \leq \ker \mathbf i}} u_{i_1j_1}\dotsb u_{i_{2k}j_{2k}} \\
= \sum_{\substack{1 \leq i_1,\dotsc,i_{l-1},i_{l+2},\dotsc i_{2k} \leq n\\ (\pi \setminus V) \leq \ker \mathbf i}} u_{i_1j_1}\dotsb u_{i_{l-1}j_{l-1}}\biggl(\sum_{i=1}^n u_{ij_l}u_{ij_{l+1}}\biggr) u_{i_{l+1}j_{l+1}}u_{i_{2k}j_{2k}}\\
= \delta_{j_l j_{l+1}} \sum_{\substack{1 \leq i_1,\dotsc,i_{l-1},i_{l+2},\dotsc i_{2k} \leq n\\ (\pi \setminus V) \leq \ker \mathbf i}} u_{i_1j_1}\dotsb u_{i_{l-1}j_{l-1}} u_{i_{l+1}j_{l+1}}u_{i_{2k}j_{2k}}
\end{multline*}
and the result follows from induction.  Plugging this in above, we find
\begin{align*}
 \sum_{1 \leq i_1,\dotsc,i_{2k} \leq n} \varphi(s_{i_1}\dotsb s_{i_{2k}})u_{i_1j_1}\dotsb u_{i_{2k}j_{2k}} &= \sum_{\substack{\pi \in NC_2(2k)\\ \pi \leq \ker \mathbf j}} \varphi(\kappa_E^{(\pi)})1_{A_o(n)}\\
&= \varphi(s_{j_1}\dotsb s_{j_{2k}})1_{A_o(n)}.
\end{align*}
Since also
\begin{align*}
 \sum_{1 \leq i_1,\dotsc,i_{2k+1} \leq n} \varphi(s_{i_1}\dotsb s_{i_{2k+1}})u_{i_1j_1}\dotsb u_{i_{2k+1}j_{2k+1}} &= \varphi(s_{j_1}\dotsb s_{j_{2k+1}})1_{A_o(n)}\\
&= 0
\end{align*}
for any $1 \leq j_1,\dotsc,j_{2k+1} \leq n$, it follows that $(s_1,\dotsc,s_n)$ is quantum rotatable.
\end{proof}

\begin{rmk}
Throughout the rest of this section, $(M,\varphi)$ will be a W$^*$-probability space and $(x_1,\dotsc,x_n)$ a sequence in $M$. We set $M_n = W^*(x_1,\dotsc,x_n)$ and $\varphi_n = \varphi|_{M_n}$.  We define
\begin{equation*}
 \eu{QR}_n = \mathrm{W}^*(\{p(x): p \in \ms P_n^{\alpha_n}\}),
\end{equation*}
where $\ms P_n^{\alpha_n}$ denotes the fixed point algebra of the coaction $\alpha_n$, i.e.
\begin{equation*}
 \ms P_n^{\alpha_n} = \{p \in \ms P_n: \alpha_n(p) = p \otimes 1\}.
\end{equation*}

\end{rmk}

\begin{prop}\label{rotcoact}
Let $(x_1,\dotsc,x_n)$ be a quantum rotatable sequence in $(M,\varphi)$.  Then there is a right coaction $\widetilde \alpha_n:M_n \to M_n \otimes \frk A_o(n)$ of the Hopf von Neumann algebra $\frk A_o(n)$ on $M_n$ determined by
\begin{equation*}
 \widetilde \alpha_n(p(x)) = (\mathrm{ev}_x \otimes \pi_{\psi_n}) \alpha_n(p)
\end{equation*}
for $p \in \ms P_n$.  Moreover, the fixed point algebra of $\widetilde \alpha_n$ is precisely $\eu{QR}_n$.
\end{prop}

\begin{proof}
Let $(\pi,\mc H,\xi)$ be the GNS representation of $\ms P_n$ for the state $\varphi_x$, and let $N = W^*(\pi(\ms P_n))$.  By \cite[Theorem 3.3]{qexcalg}, there is a right coaction $\alpha_n':N \to N \otimes \frk A_o(n)$ determined by
\begin{equation*}
 \alpha_n'(\pi(p)) = (\pi \otimes \pi_{\psi_n}) \alpha_n(p)
\end{equation*}
for $p \in \ms P_n$, and the fixed point algebra of $\alpha_n'$ is the weak closure of $\pi(\ms P_n^{\alpha_n})$.  Since $\varphi$ is a faithful state, there is a natural isomorphism $\theta:N \to M_n$ such that $\theta(\pi(p)) = p(x)$.  We can then define the coaction $\widetilde \alpha_n:M_n \to M_n \otimes \frk A_o(n)$ by
\begin{equation*}
 \widetilde \alpha_n = (\theta \otimes \mathrm{id}) \circ \alpha_n' \circ \theta^{-1},
\end{equation*}
and the result follows.
\end{proof}

\begin{rmk}
Using the invariance of the Haar state $\psi_n$, it is easily seen that the map
\begin{equation*}
 E_{\eu{QR}_n}[m] = (\mathrm{id} \otimes \psi_n)\widetilde \alpha_n(m)
\end{equation*}
is a $\varphi$-preserving conditional expectation of $M_n$ onto $\eu{QR}_n$.  We will now prove Theorem \ref{finrot} by showing that the $\eu{QR}_n$-valued distribution of $(x_1,\dotsc,x_n)$ is close to that of a $\eu{QR}_n$-valued free centered semicircular family with common variance.  First we need the following lemma.
\end{rmk}

\begin{lem}\label{rotcum}
Let $x_1,\dotsc,x_n$ be a quantum rotatable sequence in $(M,\varphi)$.  Then for any $b_0,\dotsc,b_{2k} \in \eu{QR}_n$ and $\pi \in NC_2(2k)$, we have
\begin{equation*}
\kappa_{E_{\eu{QR}_n}}^{(\pi)}[b_0x_{1}b_1 \otimes \dotsb \otimes x_{1}b_{2k}] = n^{-k}\sum_{\substack{1 \leq i_1,\dotsc,i_{2k} \leq n\\ \pi \leq \ker \mathbf i}} b_0x_{i_1}\dotsb x_{i_1}b_{2k}.
\end{equation*}
\end{lem}

\begin{proof}
The proof is by induction on $k$.  For $k = 1$, we have 
\begin{align*}
 \kappa_{E_{\eu{QR}_n}}^{(2)}[b_0x_{1}b_1\otimes x_{1}b_2] &= E_{\eu{QR}_n}[b_0x_1b_1x_1b_2] - E_{\eu{QR}_n}[b_0x_1b_1]E_{\eu{QR}_n}[x_1b_2].
\end{align*}
Now
\begin{equation*}
E_{\eu{QR}_n}[b_0x_1b_1] = \sum_{1 \leq i \leq n} b_0x_ib_1 \psi_n(u_{i1}) = 0,
\end{equation*}
so we have
\begin{align*}
 \kappa_{E_{\eu{QR}_n}}^{(\pi)}[b_0x_{1}b_1\otimes x_{1}b_2] &= E_{\eu{QR}_n}[b_0x_1b_1x_1b_2]\\
&= \sum_{1 \leq i_1,i_2 \leq n} b_0x_{i_1}b_1x_{i_2}b_2 \psi_n(u_{i_11}u_{i_21})\\
&= \frac{1}{n}\sum_{1 \leq i \leq n} b_0x_ib_1x_ib_2.
\end{align*}
If $k > 1$, let $V = \{l,l+1\}$ be an interval of $\pi$.  Then
\begin{multline*}
 n^{-k} \sum_{\substack{1 \leq i_1,\dotsc,i_{2k} \leq n\\ \pi \leq \ker \mathbf i}} b_0x_{i_1}\dotsb x_{i_{2k}}b_{2k}\\
 =  n^{-(k-1)} \sum_{\substack{1 \leq i_1,\dotsc,i_{l-1},i_{l+2},i_{2k} \leq n\\ (\pi \setminus V) \leq \ker \mathbf i}} b_0x_{i_1}\dotsb b_{l-1}\biggl(\frac{1}{n}\sum_{i=1}^n x_{i}b_lx_i\biggr)b_{l+1}\dotsb x_{i_{2k}}b_{2k}\\
= n^{-(k-1})\sum_{\substack{1 \leq i_1,\dotsc,i_{l-1},i_{l+2},i_{2k} \leq n\\ (\pi \setminus V) \leq \ker \mathbf i}} b_0x_{i_1}\dotsb b_{l-1}\kappa_{E_{\eu{QR}_n}}^{(2)}[x_1b_l\otimes x_1]b_{l+1}\dotsb x_{i_{2k}}b_{2k},
\end{multline*}
which by induction is equal to
\begin{equation*}
\kappa_{E_{\eu{QR}_n}}^{(\pi\setminus V)}[b_0x_{1}b_1 \otimes \dotsb \otimes x_{1}b_{l-1}\kappa_{E_{\eu{QR}_n}}^{(2)}[x_1b_l\otimes x_1]b_{l+1} \otimes \dotsb \otimes x_{1}b_{2k}].
\end{equation*}
But by definition this is equal to
\begin{equation*}
 \kappa_{E{\eu{QR}_n}}^{(\pi)}[b_0x_{1}b_1 \otimes \dotsb \otimes x_{1}b_{2k}],
\end{equation*}
and the result follows by induction.

\end{proof}

\begin{proof}[Proof of Theorem \ref{finrot}]
Let $(x_1,\dotsc,x_n)$ be a quantum rotatable sequence in the W$^*$-probability space $(M,\varphi)$.  Let $s_1,\dotsc,s_n$ be a $\eu{QR}_n$-valued free centered semicircular family with common variance $\eta:\eu{QR}_n \to \eu{QR}_n$ defined by
\begin{equation*}
 \eta(b) = E_{\eu{QR}_n}[x_1bx_1]
\end{equation*}
for $b \in \eu{QR}_n$.  

Let $1 \leq j_1,\dotsc,j_{2k} \leq n$, and $b_0,\dotsc,b_{2k} \in \eu{QR}_n$ with $\|b_i\| \leq 1$ for $0 \leq i \leq 2k$.  It is easily seen by induction that if $\pi \in NC_2(2k)$, $\pi \leq \ker \mathbf j$ then
\begin{equation*}
 \kappa_{E_{\eu{QR}_n}}^{(\pi)}[b_0s_{j_1}b_1 \otimes \dotsb \otimes s_{j_{2k}}b_{2k}] = \kappa_{E_{\eu{QR}_n}}^{(\pi)}[b_0 x_1b_1 \otimes \dotsb \otimes x_1b_{2k}].
\end{equation*}
It follows from Lemma \ref{rotcum} that 
\begin{align*}
 E_{\eu{QR}_n}[b_0s_{j_1}\dotsb s_{j_{2k}}b_{2k}] &= \sum_{ \substack{\pi \in NC_2(2k)\\ \pi \leq \ker \mathbf j}} \kappa_{E_{\eu{QR}_n}}^{(\pi)}[b_0s_{j_1}b_1 \otimes \dotsb \otimes s_{j_{2k}}b_{2k}]\\
&= \sum_{\substack{\pi \in NC_2(2k)\\ \pi \leq \ker \mathbf j}} \kappa_{E_{\eu{QR}_n}}^{(\pi)}[b_0x_1b_1 \otimes \dotsb \otimes x_{1}b_{2k}]\\
&= \sum_{\substack{\pi \in NC_2(2k)\\ \pi \leq \ker \mathbf j}} n^{-k}\sum_{\substack{1 \leq i_1,\dots,i_{2k} \leq n\\ \pi \leq \ker \mathbf i}} b_0x_{i_1}b_1\dotsb x_{i_{2k}}b_{2k}.
\end{align*}

On the other hand, we have
\begin{align*}
 E_{\eu{QR}_n} [b_0x_{j_1}\dotsb x_{j_{2k}}b_{2k}] &= \sum_{1 \leq i_1,\dotsb,i_{2k} \leq n} b_0x_{i_1}\dotsb x_{i_{2k}}b_{2k}\psi_n(u_{i_1j_1}\dotsb u_{i_{2k}j_{2k}})\\
&= \sum_{1 \leq i_1,\dotsb,i_{2k} \leq n} b_0x_{i_1}\dotsb x_{i_{2k}}b_{2k}\sum_{\substack{\pi,\sigma \in NC_2(2k)\\ \pi \leq \ker \mathbf i\\ \sigma \leq \ker \mathbf j}} W_{kn}(\pi,\sigma)\\
&= \sum_{\substack{\pi,\sigma \in NC_2(2k)\\ \sigma \leq \ker \mathbf j}} W_{kn}(\pi,\sigma) \sum_{\substack{1 \leq i_1,\dotsc, i_{2k} \leq n\\ \pi \leq \ker \mathbf i}} b_0x_{i_1}\dotsb x_{i_{2k}}b_{2k}.
\end{align*}
Since $x_1,\dotsc,x_n$ are identically distributed with respect to the faithful state $\varphi$, it follows that $\|x_1\| = \dotsb = \|x_n\|$.  So for any $\pi \in NC_2(2k)$,
\begin{equation*}
 \biggl\|\sum_{\substack{1 \leq i_1,\dotsc,i_{2k} \leq n\\ \pi \leq \ker \mathbf i}} b_0 x_{i_1}\dotsb x_{i_{2k}}b_{2k} \biggr\| \leq n^k \|x_1\|^{2k}.
\end{equation*}
Combining this with the equation above, we find that
\begin{multline*}
\biggl\|E_{\eu{QR}_n}[b_0x_{j_1}\dotsb x_{j_{2k}}b_{2k}] - E_{\eu{QR}_n}[b_0s_{j_1}\dotsb s_{j_{2k}}b_{2k}] \biggr\| =\\ \biggl\|\sum_{\substack{\pi,\sigma \in NC_2(2k)\\ \sigma \leq \ker \mathbf j}} (W_{kn}(\pi,\sigma) - \delta_{\pi\sigma}n^{-k})\sum_{\substack{1 \leq i_1,\dotsc,i_{2k} \leq n \\ \pi \leq \ker \mathbf i}} b_0x_{i_1}\dotsb x_{i_{2k}}b_{2k}\biggr\|\\
\leq  \sum_{\substack{\pi,\sigma \in NC_2(2k)}} |W_{kn}(\pi,\sigma)n^{k} - \delta_{\pi\sigma}|\|x_1\|^{2k} 
\end{multline*}
Setting
\begin{equation*}
D_k = \sup_{n \in \N} \; n\cdot \negthickspace\negthickspace\negthickspace\negthickspace\negthickspace  \sum_{\pi,\sigma \in NC_2(2k)} |W_{kn}(\pi,\sigma)n^{k} - \delta_{\pi\sigma}|,
\end{equation*}
which is finite by the asymptotic estimate in \ref{haar}, proves the estimate for the even moments.  For the odd moments, let $1 \leq i_1,\dotsc,i_{2k+1} \leq n$ and $b_0,\dotsc,b_{2k+1} \in \eu{QR}_n$, then
\begin{multline*}
 E_{\eu{QR}_n}[b_0x_{i_1}\dotsb x_{i_{2k+1}}b_{2k+1}]\\ = \sum_{1 \leq i_1,\dotsc,i_{2k+1} \leq n} b_0x_{i_1}\dotsb x_{i_{2k+1}}b_{2k+1}\psi_n(u_{i_1j_1}\dotsb u_{i_{2k+1}j_{2k+1}})
\end{multline*}
is equal to zero by \ref{haar}.

\end{proof}

\section{Infinite quantum rotatable sequences}

\begin{defn}
An infinite sequence $(x_i)_{i \in \N}$ of self-adjoint random variables in a noncommutative probability space $(A,\varphi)$ is called \textit{quantum rotatable} or \textit{quantum orthogonally invariant} if $(x_1,\dotsc,x_n)$ is quantum rotatable for each $n \in \N$.
\end{defn}

\begin{rmk}
This definition is equivalent to the statement that for each $n \in \N$ the joint distribution of $(x_1,\dotsc,x_n)$ is invariant under the coaction $\alpha_n$ of $\mc A_o(n)$ on $\ms P_n$ as defined in the previous section.  It will be convenient to extend these coactions to $\ms P_\infty$.
\end{rmk}

\begin{rmk}
Let $\beta_n:\ms P_\infty \to \ms P_\infty \otimes \mc A_o(n)$ be the unique unital homomorphism determined by
\begin{equation*}
 \beta_n(t_j) = \begin{cases}\sum_{i=1}^n t_i \otimes u_{ij}, & 1 \leq j \leq n\\ t_j \otimes 1, & j > n\end{cases}.
\end{equation*}
Then $\beta_n$ is a right coaction of $\mc A_o(n)$ on $\ms P_\infty$. Moreover, these coactions are compatible in the sense that
\begin{align*}
 (\mathrm{id} \otimes \omega_n) \circ \beta_{n+1} &= \beta_n \\
\intertext{and}
 (\iota_n \otimes \mathrm{id}) \circ \alpha_n &= \beta_n \circ \iota_n.
\end{align*}
where $\iota_n:\ms P_n \to \ms P_\infty$ is the obvious inclusion and $\omega_n:A_o(n+1) \to A_o(n)$ is the unique unital $*$-homomorphism, given by the universal property of $A_o(n+1)$, such that
\begin{equation*}
 \omega_n(u_{ij}) = \begin{cases} u_{ij}, & 1 \leq i,j \leq n\\ \delta_{ij}1_{A_o(n)}, & \max \{i,j\} = n+1\end{cases}.
\end{equation*}
\end{rmk}

\begin{prop}
An infinite sequence $(x_i)_{i \in \N}$ of self-adjoint elements in a noncommutative probability space $(A,\varphi)$ is quantum rotatable if and only if $\varphi_x$ is invariant under the coactions $\beta_n$ for each $n \in \N$.
\end{prop}

\begin{proof}
Let $\varphi_x^{(n)}:\ms P_n \to \C$ denote the joint distribution of $(x_1,\dotsc,x_n)$.  We have
\begin{align*}
 (\varphi_x^{(n)} \otimes \mathrm{id}) \circ \alpha_n &= (\varphi_x \circ \iota_n \otimes \mathrm{id}) \circ \alpha_n\\
&= (\varphi_x \otimes \mathrm{id}) \circ \beta_n \circ \iota_n,
\end{align*}
from which it follows that if $\varphi_x$ is invariant under $\beta_n$ then $(x_1,\dotsc,x_n)$ is quantum rotatable.  

For the converse, we note that if $\varphi_x$ is invariant under $\beta_n$ then it is invariant under $\beta_m$ for $m \leq n$.  Indeed it suffices to show that it is invariant under $\beta_{n-1}$.  Let $p \in \ms P_\infty$, then
\begin{align*}
 (\varphi_x \otimes \mathrm{id})\beta_{n-1}(p) &= (\varphi_x \otimes \mathrm{id})(\mathrm{id} \otimes \omega_{n-1}) \beta_n(p)\\
&= (\mathrm{id} \otimes \omega_{n-1})(\varphi_x(p) \otimes 1_{\mc A_o(n)})\\
&= \varphi_x(p)1_{\mc A_o(n-1)}.
\end{align*}

Now suppose that $\varphi_x^{(n)}$ is invariant under $\alpha_n$ for each $n \in \N$.  Let $m \in \N$ and $p \in \ms P_\infty$, then $p = \iota_n(p')$ for some $p' \in \ms P_n$, $n \geq m$.  We then have
\begin{align*}
 (\varphi_x \otimes \mathrm{id}) \beta_n(p) &= (\varphi_x^{(n)} \otimes \mathrm{id}) \alpha_n(p')\\
&= \varphi_x(p)1_{\mc A_o(n)}.
\end{align*}

\end{proof}

\begin{rmk}
Throughout the rest of the section, $(M,\varphi)$ will be a W$^*$-probability space, and $(x_i)_{i \in \N}$ a sequence of self-adjoint random variables in $M$.  $M_\infty$ will denote the von Neumann algebra generated by $\{x_i:i \in \N\}$.  By a slight abuse of notation, we denote
\begin{equation*}
 \eu{QR}_n = \mathrm{W}^*(\{p(x): p \in \ms P_\infty^{\beta_n}\})
\end{equation*}
where $\ms P_\infty^{\beta_n}$ denotes the fixed point algebra of the coaction $\beta_n$.  Since
\begin{equation*}
 (\textrm{id} \otimes \omega_n) \circ \beta_{n+1} = \beta_n,
\end{equation*}
it follows that $\eu{QR}_{n+1} \subset \eu{QR}_n$ for all $n \geq 1$.  We then define
\begin{equation*}
 \eu{QR} = \bigcap_{n \geq 1} \eu{QR}_n.
\end{equation*}
\end{rmk}

\begin{rmk}
If $(x_i)_{i \in \N}$ is quantum rotatable, then it follows as in Proposition \ref{rotcoact} that for each $n \in \N$ the coaction $\beta_n$ lifts to a right coaction $\widetilde \beta_n:M_\infty \to M_\infty \otimes \frk A_o(n)$ of the Hopf von Neumann algebra $\frk A_o(n)$ on $M_\infty$ determined by
\begin{equation*}
 \widetilde \beta_n(p(x)) = (\mathrm{ev}_x \otimes \pi_{\psi_n}) \beta_n(p)
\end{equation*}
for $p \in \ms P_\infty$, and moreover the fixed point algebra of $\widetilde \beta_n$ is $\eu{QR}_n$.  For each $n \in \N$, there is a $\varphi$-preserving conditional expectation $E_{\eu{QR}_n}$ of $M_\infty$ onto $\eu{QR}_n$ given by integrating $\beta_n$, i.e.
\begin{equation*}
 E_{\eu{QR}_n}[m] = (\mathrm{id} \otimes \psi_n)\beta_n(m)
\end{equation*}
for $m \in M_\infty$.  As the next proposition shows, we may obtain a $\varphi$-preserving conditional expectation onto $\eu{QR}$ by taking the limit as $n$ goes to infinity.  Since we will need a similar result in the quantum unitary case, we will give a more general statement.  The proof is the same as \cite[Proposition 5.7]{qexcalg}, but is included for the convenience of the reader.
\end{rmk}

\begin{prop}\label{explim}
Let $(M,\varphi)$ be a W$^*$-probability space, and for each $n \in \N$ let $1 \in B_n \subset M$ be a W$^*$-subalgebra.  Suppose that $B_{n+1} \subset B_n$ for each $n \in \N$ and set
\begin{equation*}
 B = \bigcap_{n \geq 1} B_n.
\end{equation*}
Suppose further that for each $n \in \N$, there is a $\varphi$-preserving conditional expectation $E_n:M \to B_n$. Then 
\begin{enumerate}
\item For any $m \in M$, the sequence $E_n[m]$ converges in $|\;|_2$ and the strong topology to a limit $E[m]$ in $B$.  Moreover, $E$ is a $\varphi$-preserving conditional expectation of $M$ onto $B$.
\item Fix $\pi \in NC(k)$ and $m_1,\dotsc,m_k \in M$, then
\begin{equation*}
 E^{(\pi)}[m_1 \otimes \dotsb \otimes m_k] = \lim_{n \to \infty} E_{n}^{(\pi)}[m_1 \otimes \dotsb \otimes m_k],
\end{equation*}
with convergence in the strong topology.
\end{enumerate}

\end{prop}

\begin{proof}
Let $\phi_n = \varphi|_{B_n}$ and let $L^2(B_n,\phi_n)$ denote the GNS Hilbert space, which can be viewed as a closed subspace of $L^2(M,\varphi)$.  Let $P_n \in \mc B(L^2(M,\varphi))$ be the orthogonal projection onto $L^2(B_n,\phi_n)$.  Since $E_{n}:M \to B_n$ is a conditional expectation such that $\phi_n \circ E_{n} = \varphi$, it follows (see e.g. \cite[Proposition II.6.10.7]{blackadar}) that 
\begin{equation*}
 E_{\eu{QE}_n}[m] = P_n m P_n
\end{equation*}
for $m \in M$.  Since $P_n$ converges strongly as $n \to \infty$ to $P$, where
\begin{equation*}
 P = \bigwedge_{n \geq 1} P_n
\end{equation*}
is the orthogonal projection onto $L^2(B,\varphi|_{B})$, it follows that 
\begin{equation*}
 E_{n}[m] \to PmP
\end{equation*}
in $|\;|_2$ and the strong operator topology as $n \to \infty$.  Set $E[m] = PmP$, then since $E_{n}[m]$ converges strongly to $E[m]$ it follows that $E[m] \in B$, and it is then easy to see that $E$ is a $\varphi$-preserving conditional expectation.

To prove (ii), observe that if $\pi \in NC(k)$ and $m_1,\dotsc,m_k \in M$, then $E_{n}^{(\pi)}[m_1 \otimes \dotsb \otimes m_k]$ is a word in $m_1,\dotsc,m_k$ and $P_n$.  For example, if 
\begin{equation*} 
\pi = \{\{1,10\},\{2,5,6\},\{3,4\}, \{7,8,9\}\} \in NC(10),
\end{equation*}
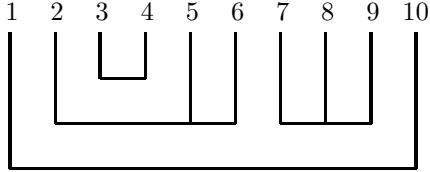
\begin{equation*}
 \setlength{\unitlength}{0.6cm} \begin{picture}(9,4)\thicklines \put(0,0){\line(0,1){3}}
\put(0,0){\line(1,0){9}} \put(9,0){\line(0,1){3}} \put(8,1){\line(0,1){2}} \put(7,1){\line(0,1){2}}
\put(6,1){\line(0,1){2}} \put(6,1){\line(1,0){2}}
\put(1,1){\line(1,0){4}} \put(1,1){\line(0,1){2}} 
\put(2,2){\line(1,0){1}} \put(2,2){\line(0,1){1}} \put(3,2){\line(0,1){1}} 
\put(5,1){\line(0,1){2}} \put(4,1){\line(0,1){2}}
\put(-0.1,3.3){1} \put(0.9,3.3){2} \put(1.9,3.3){3}
\put(2.9,3.3){4} \put(3.9,3.3){5} \put(4.9,3.3){6} \put(5.9,3.3){7} \put(6.9,3.3){8}
\put(7.9,3.3){9} \put(8.7,3.3){10}
\end{picture}
\end{equation*}
then the corresponding expression is
\begin{equation*}
 E_{n}^{(\pi)}[m_1 \otimes \dotsb \otimes m_{10}] =  P_nm_1Pm_2Pm_3m_4P_nm_5m_6P_nm_7m_8m_9P_nm_{10}P_n.
\end{equation*}
Since multiplication is jointly continuous on bounded sets in the strong topology, this converges as $n$ goes to infinity to the expression obtained by replacing $P_n$ by $P$, which is exactly $E^{(\pi)}[m_1 \otimes \dotsb \otimes m_k]$.

\end{proof}

\begin{rmk}
With these preparations we pass to the proof of Theorem \ref{infrot}.
\end{rmk}

\begin{proof}[Proof of Theorem \ref{infrot}]
The implication (ii) $\Rightarrow$ (i) follows from Proposition \ref{semrot}.  Let $(x_i)_{i \in \N}$ be a quantum rotatable sequence in the W$^*$-probability space $(M,\varphi)$.  Let $j_1,\dotsc,j_{2k} \in \N$ and $b_0,\dotsc,b_{2k} \in \eu{QR}$.  As in the proof of Theorem \ref{finrot}, we have
\begin{align*}
 E_{\eu{QR}}[b_0x_{j_1}\dotsb x_{j_{2k}}b_{2k}] &= \lim_{n \to \infty} E_{\eu{QR}_n}[b_0x_{j_1}\dotsb x_{j_{2k}}b_{2k}]\\
&= \lim_{n \to \infty}\sum_{\substack{\pi,\sigma \in NC_2(2k)\\ \sigma \leq \ker \mathbf j}} W_{kn}(\pi,\sigma) \sum_{\substack{1 \leq i_1,\dotsc,i_{2k} \leq n\\ \pi \leq \ker \mathbf i}}b_0x_{i_1}\dotsb x_{i_{2k}}b_{2k},
\end{align*}
with convergence in the strong topology.  Moreover, for any $\pi,\sigma \in NC_2(2k)$
\begin{equation*}
 \lim_{n \to \infty} |W_{kn}(\pi,\sigma)-\delta_{\pi,\sigma}n^{-k}|\sum_{\substack{1 \leq i_1,\dotsc,i_{2k} \leq n\\ \pi \leq \ker \mathbf i}} \|b_0x_{i_1}\dotsb x_{i_{2k}}b_{2k}\| = 0,
\end{equation*}
from which it follows that 
\begin{equation*}
  E_{\eu{QR}}[b_0x_{j_1}\dotsb x_{j_{2k}}b_{2k}] = \lim_{n \to \infty} \sum_{ \substack{\pi \in NC_2(2k)\\ \pi \leq \ker \mathbf j}} n^{-k}\sum_{\substack{1 \leq i_1,\dotsc,i_{2k}\leq n\\ \pi \leq \ker \mathbf i}} b_0x_{i_1}\dotsb x_{i_{2k}}b_{2k}.
\end{equation*}
By Lemma \ref{rotcum}, for $\pi \in NC_2(2k)$,  we have
\begin{equation*}
 n^{-k}\sum_{\substack{1 \leq i_1,\dotsc,i_{2k} \leq n\\ \pi \leq \ker \mathbf i}} b_0x_{i_1}\dotsb x_{i_{2k}}b_{2k} = \kappa_{E_{\eu{QR}_n}}^{(\pi)}[b_0x_{1}b_1 \otimes \dotsb \otimes x_{1}b_{2k}].
\end{equation*}
By Proposition \ref{explim},
\begin{equation*}
 \lim_{n \to \infty} \kappa_{E_{\eu{QR}_n}}^{(\pi)}[b_0x_{1}b_1 \otimes \dotsb \otimes x_{1}b_{2k}] = \kappa_{E_{\eu{QR}}}^{(\pi)}[b_0x_{1}b_1 \otimes \dotsb \otimes x_{1}b_{2k}].
\end{equation*}
Plugging this in above, we have
\begin{equation*}
 E_{\eu{QR}}[b_0x_{j_1}\dotsb x_{j_{2k}}b_{2k} ] = \sum_{ \substack{\pi \in NC_2(2k)\\ \pi \leq \ker \mathbf j}} \kappa_{E_{\eu{QR}}}^{(\pi)}[b_0x_{1}b_1 \otimes \dotsb \otimes x_{1}b_{2k}].
\end{equation*}
It follows from Theorem \ref{finrot} that the odd moments
\begin{equation*}
 E_{\eu{QR}}[b_0x_{j_1}\dotsb x_{j_{2k+1}}b_{2k+1}]
\end{equation*}
are zero for any $j_1,\dotsc,j_{2k+1} \in \N$ and $b_0,\dotsc,b_{2k+1} \in \eu{QR}$, and the result now follows from Proposition \ref{semcircexp}.
\end{proof}

\begin{rmk}\label{finex}
We will now give an example which demonstrates that Theorem \ref{infrot} fails for finite sequences.  Consider the sequence $x_j = \pi(u_{1j})$ for $1 \leq j \leq n$ in the W$^*$-probability space $(\frk A_o(n),\psi_n)$.  That the sequence is quantum rotatable is simply the invariance condition of the Haar state $\psi_n$.  We will show that $(x_1,\dotsc,x_n)$ is not freely independent and identically distributed with respect to any $\psi_n$-preserving conditional expectation $E$.  Suppose that it were.  The orthogonality relation in $A_o(n)$ gives
\begin{equation*}
 \sum_{i=1}^n x_i^2 = 1,
\end{equation*}
which implies that $E[x_i^2] = 1/n$ for $1 \leq i \leq n$.  Squaring this relation and applying $\psi$ gives
\begin{align*}
 \sum_{1 \leq i,j \leq n} E[x_i^2 x_j^2] = 1.
\end{align*}
Since $(x_1,\dotsc,x_n)$ are assumed to be free and identically distributed with respect to $E$, this becomes
\begin{equation*}
 n(n-1)E[x_1^2]^2 + nE[x_1^4] = 1,
\end{equation*}
from which it follows that
\begin{equation*}
 E[x_1^4] = \frac{1}{n^2}.
\end{equation*}
Applying $\psi_n$, we find
\begin{equation*}
 \psi_n(x_1^4) = \frac{1}{n^2} = \psi_n(x_1^2)^2.
\end{equation*}
Since $x_1^2$ is positive and $\psi_n$ is faithful, this implies $x_1^2 = \frac{1}{n}$ which is absurd.  So $(x_1,\dotsc,x_n)$ are not freely independent and identically distributed with respect to a $\psi_n$-preserving conditional expectation.
\end{rmk}

\section{Quantum unitary invariance}

\noindent In this section we define quantum unitary invariance for a sequence of noncommutative random variables, and prove Theorem \ref{infunit}.  The approach is similar to the quantum orthogonal case, and some details are left to the reader.
\begin{rmk}
Let $\beta_n:\ms Q_\infty \to \ms Q_\infty \otimes \mc A_u(n)$ be the unique unital $*$-homomorphism determined by
\begin{equation*}
 \beta_n(t_j) = \begin{cases} 
\sum_{i=1}^n t_i \otimes v_{ij}, & 1 \leq j \leq n\\
t_j \otimes 1, & j > n
\end{cases}.
\end{equation*}
It is easily seen that $\beta_n$ is a right coaction of the Hopf $*$-algebra $\mc A_u(n)$ on $\ms Q_n$.
\end{rmk}

\begin{defn}
If $(x_i)_{i \in \N}$ is a sequence of (not necessarily self-adjoint) random variables in a noncommutative probability space $(A,\varphi)$, we say that $\varphi_x$ is \textit{invariant under quantum unitary transformations}, or that the sequence is \textit{quantum unitarily invariant}, if $\varphi_x$ is invariant under $\beta_n$ for every $n \in \N$.
\end{defn}

\begin{rmk}
First we will show that operator-valued free centered circular families with common variance are quantum unitarily invariant.
\end{rmk}

\begin{prop}\label{circunit}
Let $A$ be a unital algebra, $1 \in B \subset A$ a $*$-subalgebra and $E:A \to B$ a conditional expectation which preserves $\varphi$.  Suppose that $(c_i)_{i \in \N}$ is a $B$-valued free centered circular family with common variance.  Then $(c_i)_{i \in \N}$ is quantum unitarily invariant.
\end{prop}

\begin{proof}
Let $1 \leq j_1,\dotsc,j_{2k} \leq n$ and $d_1,\dotsc,d_{2k} \in \{1,*\}$, then as in the proof of Proposition \ref{semrot} we have
\begin{multline*}
 \sum_{1 \leq i_1,\dotsc,i_{2k} \leq n} \varphi(c_{i_1}^{d_1}\dotsb c_{i_{2k}}^{d_k})v_{i_1j_1}\dotsb v_{i_{2k}j_{2k}} =\\ \sum_{\pi \in NC_2^{\mathbf d}(2k)} \varphi(\kappa_E^{(\pi)}[c_1^{d_1} \otimes \dotsb \otimes c_1^{d_{2k}}])\sum_{\substack{1 \leq i_1,\dotsc,i_{2k} \leq n\\ \pi \leq \ker \mathbf i}} v_{i_1j_1}^{d_1}\dotsb v_{i_{2k}j_{2k}}^{d_{2k}}.
\end{multline*}
An inductive argument similar to that given in Proposition \ref{semrot} shows that
\begin{equation*}
 \sum_{\substack{1 \leq i_1,\dotsc,i_{2k} \leq n\\ \pi \leq \ker \mathbf i}} v_{i_1j_1}^{d_1}\dotsb v_{i_{2k}j_{2k}}^{d_{2k}} = \begin{cases} 1_{A_u(n)}, & \pi \leq \ker \mathbf j\\ 0, & \text{otherwise}\end{cases}
\end{equation*}
for any $\pi \in NC_2^{\mathbf d}(2k)$.  It follows that that
\begin{align*}
 \sum_{1 \leq i_1,\dotsc,i_{2k} \leq n} \varphi(c_{i_1}^{d_1}\dotsb c_{i_{2k}}^{d_{2k}})v_{i_1j_1}\dotsb v_{i_{2k}j_{2k}} &= \sum_{\substack{\pi \in NC_2^{\mathbf d}(2k)\\ \pi \leq \ker \mathbf j}} \negthickspace \varphi(\kappa_E^{(\pi)}[c_1^{d_1} \otimes \dotsb \otimes c_1^{d_{2k}}])1_{A_u(n)}\\
&= \varphi(c_{j_1}^{d_1}\dotsb c_{j_{2k}}^{d_{2k}})1_{A_u(n)}.
\end{align*}
Since also
\begin{equation*}
 \sum_{1 \leq i_1,\dotsc,i_{2k+1} \leq n} \varphi(c_{i_1}^{d_1}\dotsb c_{i_{2k+1}}^{d_{2k+1}})v_{i_1j_1}\dotsb v_{i_{2k+1}j_{2k+1}} = 0 = \varphi(c_{j_1}^{d_1}\dotsb c_{j_{2k+1}}^{d_{2k+1}}) 1_{A_u(n)}
\end{equation*}
for any $1 \leq j_1,\dotsc,j_{2k+1} \leq n$ and $d_1,\dotsc,d_{2k+1} \in \{1,*\}$, it follows that $(c_i)_{i \in \N}$ is quantum unitarily invariant as claimed.
\end{proof}

\begin{rmk}
Throughout the rest of the section, $(M,\varphi)$ will be a W$^*$-probability space and $(x_i)_{i \in \N}$ a sequence in $M$.  As in the previous section, $M_\infty$ will denote the von Neumann algebra generated by $\{x_i:i \in \N\}$.  We denote
\begin{equation*}
 \eu{QU}_n = \mathrm{W}^*(\{q(x):q \in \ms Q_\infty^{\beta_n}\}),
\end{equation*}
where $\ms Q_\infty^{\beta_n}$ is the fixed point algebra of the coaction $\beta_n$.  We then set
\begin{equation*}
 \eu{QU} = \bigcap_{n \geq 1} \eu{QU}_n.
\end{equation*}

As in the orthogonal case, if $(x_i)_{i \in \N}$ is quantum unitarily invariant sequence then there is a right coaction $\widetilde \beta_n:M_\infty \to M_\infty \otimes \frk A_u(n)$ of the Hopf von Neumann algebra $\frk A_u(n)$ on $M_n$ which is determined by
\begin{equation*}
 \widetilde \alpha_n(q(x)) = (\mathrm{ev}_x \otimes \pi_{\psi_n}) \alpha_n(q)
\end{equation*}
for $q \in \ms Q_n$, and the fixed point algebra of this coaction is $\eu{QU}_n$.  There is then a $\varphi$-preserving conditional expectation $E_{\eu{QU}_n}$ of $M_\infty$ onto $\eu{QU}_n$ given by
\begin{equation*}
 E_{\eu{QU}_n}[m] = (\mathrm{id} \otimes \psi_n) \widetilde \alpha_n(m)
\end{equation*}
for $m \in M_\infty$.
\end{rmk}

\begin{rmk}
To prove Theorem \ref{infunit}, we will first need the following result.
\end{rmk}

\begin{lem}\label{unitcum}
Let $(x_i)_{i \in \N}$ be a quantum unitarily invariant sequence in $(M,\varphi)$.  Then for any $b_0,\dotsc,b_{2k} \in \eu{QU}_n$, $d_1,\dotsc,d_{2k} \in \{1,*\}$ and $\pi \in NC_2^\mathbf d(2k)$, we have
\begin{equation*}
 \kappa_{E_{\eu{QU}}}^{(\pi)}[b_0x_{1}^{d_1}b_1 \otimes \dotsb \otimes x_{1}^{d_{2k}}b_{2k}] = \lim_{n \to \infty} n^{-k} \sum_{\substack{1 \leq i_1,\dotsc,i_{2k} \leq n\\ \pi \leq \ker \mathbf i}} b_0x_{i_1}^{d_1}\dotsb x_{i_{2k}}^{d_{2k}}b_{2k},
\end{equation*}
with convergence in the strong topology.
\end{lem}

\begin{proof}
By Proposition \ref{explim}, we have
\begin{equation*}
 \kappa_{E_{\eu{QU}}}^{(\pi)}[b_0x_{1}^{d_1}b_1 \otimes \dotsb \otimes x_{1}^{d_{2k}}b_{2k}] = \lim_{n \to \infty} \kappa_{E_{\eu{QU}_n}}^{(\pi)}[b_0x_{1}^{d_1}b_1 \otimes \dotsb \otimes x_{1}^{d_{2k}}b_{2k}].
\end{equation*}
It therefore suffices to show that
\begin{equation*}
 \kappa_{E_{\eu{QU}_n}}^{(\pi)}[b_0x_{1}^{d_1}b_1 \otimes \dotsb \otimes x_{1}^{d_{2k}}b_{2k}] = n^{-k} \sum_{\substack{1 \leq i_1,\dotsc,i_{2k} \leq n\\ \pi \leq \ker \mathbf i}} b_0x_{i_1}^{d_1}\dotsb x_{i_{2k}}^{d_{2k}}b_{2k}.
\end{equation*}
This is proved by an inductive argument similar to that given for Lemma \ref{rotcum}.
\end{proof}

\begin{proof}[Proof of Theorem \ref{infunit}]
The implication (ii) $\Rightarrow$ (i) follows from Proposition \ref{circunit}.  Let $(x_i)_{i \in \N}$ be a quantum unitarily invariant sequence in the W$^*$-probability space $(M,\varphi)$.
Let $j_1,\dotsc,j_{2k} \in \N$, $b_0,\dotsc,b_{2k} \in \eu{QU}$, and $d_1,\dotsc,d_{2k} \in \{1,*\}$.  As in the proof of Theorem \ref{infrot}, we have
\begin{align*}
 E_{\eu{QU}}[b_0x_{j_1}^{d_1}\dotsb x_{j_{2k}}^{d_{2k}}b_{2k}] &= \lim_{n \to \infty} E_{\eu{QU}_n}[b_0x_{j_1}^{d_1}\dotsb x_{j_{2k}}^{d_{2k}}b_{2k}]\\
&= \lim_{n \to \infty}\sum_{\substack{\pi,\sigma \in NC_2^\mathbf d(2k)\\ \sigma \leq \ker \mathbf j}} W_{kn}(\pi,\sigma) \sum_{\substack{1 \leq i_1,\dotsc,i_{2k} \leq n\\ \pi \leq \ker \mathbf i}}b_0x_{i_1}^{d_1}\dotsb x_{i_{2k}}^{d_{2k}}b_{2k}\\
&= \lim_{n \to \infty} \sum_{ \substack{\pi \in NC_2^\mathbf d(2k)\\ \pi \leq \ker \mathbf j}} n^{-k}\sum_{\substack{1 \leq i_1,\dotsc,i_{2k} \leq n\\ \pi \leq \ker \mathbf i}} b_0x_{i_1}^{d_1}\dotsb x_{i_{2k}}^{d_{2k}}b_{2k}.
\end{align*}
Applying Lemma \ref{unitcum}, we have
\begin{equation*} 
 E_{\eu{QU}}[b_0x_{j_1}^{d_1}\dotsb x_{j_{2k}}^{d_{2k}}b_{2k}] = \sum_{ \substack{\pi \in NC_2^{\mathbf d}(2k)\\ \pi \leq \ker \mathbf j}} \kappa_{E_{\eu{QU}}}^{(\pi)}[b_0x_{1}^{d_1}b_1 \otimes \dotsb \otimes x_{1}^{d_{2k}}b_{2k}].
\end{equation*}
It is easy to see that the odd moments are zero, and the result then follows from Proposition \ref{semcircexp}.
\end{proof}

\begin{rmk}
Using the approach in Section 3, one may obtain an approximation result for finite quantum unitarily invariant sequences similar to Theorem \ref{finrot}.  The details are left to the reader.
\end{rmk}

\section*{Acknowledgement}

I would like to thank Dan-Virgil Voiculescu for his continued guidance and support while working on this project.

\bibliographystyle{hsiam}
\bibliography{ref}

\begin{thebibliography}{10}

\bibitem{bbc}
{\sc T.~Banica, J.~Bichon, and B.~Collins}, {\em Quantum permutation groups:
  {A} survey}, Banach Center Publ.,  (2007), pp.~13--34.

\bibitem{bc}
{\sc T.~Banica and B.~Collins}, {\em {Integration over compact quantum
  groups}}, Publ. Res. Inst. Math. Sci., 43 (2007), pp.~277--302.

\bibitem{blackadar}
{\sc B.~Blackadar}, {\em Operator Algebras: Theory of C$^*$-algebras and von
  Neumann Algebras}, vol.~122 of Encyclopaedia of Mathematical Sciences,
  Springer-Verlag, 2006.

\bibitem{qexcalg}
{\sc S.~Curran}, {\em Quantum exchangeable sequences of algebras}, Indiana
  Univ. Math. J.
\newblock to appear.

\bibitem{pd1}
{\sc P.~Diaconis and D.~Freedman}, {\em {Finite exchangeable sequences.}}, Ann.
  Probab., 8 (1980), pp.~745--764.

\bibitem{pd2}
{\sc P.~Diaconis and D.~Freedman}, {\em {A dozen de Finetti-style results in
  search of a theory.}}, Ann. Inst. H. Poincar´e Probab. Statist., 23 (1987),
  pp.~397--423.

\bibitem{pd3}
{\sc P.~Diaconis and D.~Freedman}, {\em {Conditional limit theorems for
  exponential families and finite versions of de Finetti's theorem.}}, J.
  Theor. Probab., 1 (1988), pp.~381--410.

\bibitem{bcircular}
{\sc K.~Dykema and G.~Tucci}, {\em {Hyperinvariant subspaces for some
  B--circular operators}}, Mathematische Annalen, 333 (2005), pp.~485--523.

\bibitem{freedman}
{\sc D.~Freedman}, {\em {Invariants under mixing which generalize de Finetti's
  theorem.}}, Ann. Math. Stat., 33 (1962), pp.~916--923.

\bibitem{kallenberg}
{\sc O.~Kallenberg}, {\em {Probabilistic Symmetries And Invariance
  Principles}}, Probability and Its Applications, Springer, 2005.

\bibitem{spekos}
{\sc C.~K{\"o}stler and R.~Speicher}, {\em A noncommutative de {F}inetti
  theorem: Invariance under quantum permutations is equivalent to freeness with
  amalgamation}, Comm. Math. Phys.
\newblock to appear.

\bibitem{nica}
{\sc A.~Nica}, {\em R-transforms of free joint distributions and non-crossing
  partitions}, J. Funct. Anal., 135 (1996), pp.~271--296.

\bibitem{ns}
{\sc A.~Nica and R.~Speicher}, {\em Lectures on the Combinatorics of Free
  Probability}, no.~335 in London Mathematical Society Lecture Note Series,
  Cambridge University Press, 2006.

\bibitem{memoir}
{\sc R.~Speicher}, {\em {Combinatorial theory of the free product with
  amalgamation and operator-valued free probability theory.}}, Mem. Am. Math.
  Soc., 627 (1998), p.~88.

\bibitem{opadd}
{\sc D.~Voiculescu}, {\em Operations on certain non-commutative operator-valued
  random variables}, Ast{\'e}risque,  (1995), pp.~243--275.

\bibitem{vdn}
{\sc D.~Voiculescu, K.~Dykema, and A.~Nica}, {\em Free Random Variables},
  vol.~1 of CRM Monograph Series, American Mathematical Society, Rhode Island,
  1992.

\bibitem{qorthunit}
{\sc S.~Wang}, {\em Free products of compact quantum groups}, Comm. Math.
  Phys., 167 (1995), pp.~671--692.

\bibitem{qpermutations}
{\sc S.~Wang}, {\em Quantum symmetry groups of finite spaces}, Comm. Math.
  Phys., 195 (1998), pp.~195--211.

\bibitem{woronowicz}
{\sc S.~Woronowicz}, {\em Compact matrix pseudogroups}, Comm. Math. Phys., 111
  (1987), pp.~613--665.

\end{thebibliography}

\end{document}